\documentclass[english]{article}
\usepackage[T1]{fontenc}
\usepackage[latin9]{inputenc}
\usepackage{amsthm}
\usepackage{amsmath}
\usepackage{amssymb}
\usepackage{esint}

\makeatletter
  \theoremstyle{plain}
  \newtheorem*{thm*}{\protect\theoremname}
\theoremstyle{plain}
\newtheorem{thm}{\protect\theoremname}[section]
  \theoremstyle{plain}
  \newtheorem{lem}[thm]{\protect\lemmaname}
  \theoremstyle{plain}
  \newtheorem{cor}[thm]{\protect\corollaryname}
  \theoremstyle{remark}
  \newtheorem{rem}[thm]{\protect\remarkname}
\newcommand{\lyxaddress}[1]{
\par {\raggedright #1
\vspace{1.4em}
\noindent\par}
}

\@ifundefined{date}{}{\date{}}
\makeatother

\usepackage{babel}
  \providecommand{\corollaryname}{Corollary}
  \providecommand{\lemmaname}{Lemma}
  \providecommand{\remarkname}{Remark}
  \providecommand{\theoremname}{Theorem}
\providecommand{\theoremname}{Theorem}

\begin{document}
\global\long\def\ric{\mathrm{Ric}}

\global\long\def\biric{\mathrm{BiRic}}

\global\long\def\vol{\mathrm{vol}}

\global\long\def\dist{\mathrm{dist}}

\title{Vanishing theorems for $L^{2}$ harmonic forms on complete Riemannian
manifolds}

\author{Matheus Vieira}
\maketitle
\begin{abstract}
This paper contains some vanishing theorems for $L^{2}$ harmonic
forms on complete Riemannian manifolds with a weighted Poincaré inequality
and a certain lower bound of the curvature. The results are in the
spirit of Li-Wang and Lam, but without assumptions of sign and growth
rate of the weight function, so they can be applied to complete stable
hypersurfaces.
\end{abstract}

\section{Introduction}

It is an interesting problem in geometry and topology to find sufficient
conditions for the space of harmonic $k$-forms to be trivial.

For compact manifolds, by Hodge theory, the space of harmonic $k$-forms
is isomorphic to the $k$-th de Rham cohomology group. In particular,
if the space of harmonic $k$-forms is trivial then the $k$-th de
Rham cohomology group is trivial.

For complete Riemannian manifolds it is natural to consider $L^{2}$
harmonic forms. Even though the space of $L^{2}$ harmonic $k$-forms
is not necessarily isomorphic to the $k$-th de Rham cohomology group,
the theory of $L^{2}$ harmonic one-forms can be used to study the
topology at infinity. Li and Tam proved that for a complete Riemannian
manifold if the space of $L^{2}$ harmonic one-forms is trivial then
the manifold has at most one non-parabolic end \cite{LT1992}. In
particular, if the space of $L^{2}$ harmonic one-forms is trivial
and there are no parabolic ends then the manifold is connected at
infinity. It is well known that certain geometric conditions imply
the non-existence of parabolic ends, for example when the manifold
satisfies a certain Sobolev inequality \cite[Lemma 20.12]{LBOOK}.
This shows that the vanishing of $L^{2}$ harmonic one-forms is closely
related to the topology at infinity of complete Riemannian manifolds.

For a complete Riemannian manifold $M$ recall that the first eigenvalue
of the Laplacian is defined by
\[
\lambda_{1}\left(M\right)=\inf_{\phi\in C_{c}^{\infty}\left(M\right)}\frac{\int_{M}\left|\nabla\phi\right|^{2}}{\int_{M}\phi^{2}}.
\]

Li and Wang proved the following vanishing theorem for $L^{2}$ harmonic
one-forms on complete Riemannian manifolds with positive spectrum.
\begin{thm*}[Li-Wang \cite{LW2001}]
 If a complete Riemannian manifold $M^{n}$ has positive spectrum
$\lambda_{1}\left(M^{n}\right)>0$ and the Ricci curvature satisfies
\[
\ric\geq-a\lambda_{1}\left(M^{n}\right)
\]
for some $0<a<\frac{n}{n-1}$, then the space of $L^{2}$ harmonic
one-forms on $M^{n}$ is trivial.
\end{thm*}
The papers \cite{LW2001,LW2002,LW2005,KLZ2008,LW2009} contain several
interesting results concerning the geometry and topology at infinity
of complete Riemannian manifolds with positive spectrum.

For a complete Riemannian manifold $M$ and a continuous function
$q$ on $M$ recall that $M$ satisfies a weighted Poincaré inequality
with weight function $q$ if 
\begin{equation}
\int_{M}q\phi^{2}\leq\int_{M}\left|\nabla\phi\right|^{2}\label{eq:wpi}
\end{equation}
for all functions $\phi$ in $C_{c}^{\infty}\left(M\right)$ \cite{LW2006}.
Taking $q=\lambda_{1}\left(M\right)$ in this definition recovers
the first eigenvalue of the Laplacian.

Li and Wang's theorem was recently generalized by Lam. He proved the
following vanishing theorem for $L^{2}$ harmonic one-forms on complete
Riemannian manifolds with a weighted Poincaré inequality.
\begin{thm*}[Lam \cite{L2010}]
 Suppose a complete Riemannian manifold $M^{n}$ satisfies a weighted
Poincaré inequality with weight function $q$ and the Ricci curvature
satisfies
\[
\ric\geq-aq
\]
for some $0<a<\frac{n}{n-1}$. Assume $q$ positive with growth rate
\[
q\left(x\right)=O\left(\dist\left(x,x_{0}\right)^{2-\alpha}\right)
\]
for some $0<\alpha<2$. Then the space of $L^{2}$ harmonic one-forms
on $M^{n}$ is trivial.
\end{thm*}
Taking $q=\lambda_{1}\left(M\right)$ in this theorem recovers Li
and Wang's theorem.

It is well known that a stable hypersurface $M^{n}$ in a Riemannian
manifold $\overline{M}^{n+1}$ satisfies a weighted Poincaré inequality
with weight function
\[
q=\left|A\right|^{2}+\overline{\ric}\left(\nu,\nu\right),
\]
where $A$ is the second fundamental form and $\overline{\ric}\left(\nu,\nu\right)$
is the Ricci curvature of $\overline{M}^{n+1}$ in the normal direction.
Under certain natural conditions it is possible to show that the Ricci
curvature of $M^{n}$ satisfies
\[
\ric\geq-aq
\]
for some $0<a<\frac{n}{n-1}$ (see Section 3). When Lam's theorem
is applied to $M^{n}$ the assumption of $q$ positive with growth
rate $q\left(x\right)=O\left(\dist\left(x,x_{0}\right)^{2-\alpha}\right)$
for some $0<\alpha<2$ is not a natural condition. This example from
hypersurface theory shows the importance of studying weighted Poincaré
inequalities without assumptions of the weight function.

The following theorem improves Lam's theorem by removing the assumptions
of sign and growth rate of the weight function.
\begin{thm}
\emph{\label{vanish cor}}Suppose a complete non-compact Riemannian
manifold $M^{n}$ satisfies a weighted Poincaré inequality with weight
function $q$ and the Ricci curvature satisfies 
\[
\ric\geq-aq
\]
for some $0<a<\frac{n}{n-1}$. Then the space of $L^{2}$ harmonic
one-forms on $M^{n}$ is trivial.
\end{thm}
This theorem will be used to prove the vanishing of $L^{2}$ harmonic
one-forms on complete stable minimal hypersurfaces in Riemannian manifolds
with non-negative Bi-Ricci$^{a}$ curvature (Theorem \ref{vanish applic}).
When the Riemannian manifold has dimension at most $7$ a similar
result holds for hypersurfaces not necessarily minimal (Theorem \ref{vanish applic 2}).

The following theorem shows that Theorem \ref{vanish cor} holds with
a more general lower bound of the Ricci curvature when the first eigenvalue
of the Laplacian satisfies a certain lower bound.
\begin{thm}
\label{lambda cor}Suppose a complete Riemannian manifold $M^{n}$
satisfies a weighted Poincaré inequality with weight function $q$
and the Ricci curvature satisfies 
\[
\ric\geq-aq-b
\]
for some $0<a<\frac{n}{n-1}$ and $b>0$. Assume the first eigenvalue
of the Laplacian satisfies
\[
\lambda_{1}\left(M^{n}\right)>\frac{b}{\frac{n}{n-1}-a}.
\]
Then the space of $L^{2}$ harmonic one-forms on $M^{n}$ is trivial.
\end{thm}
This theorem can be used as following: with the same assumption of
the Ricci curvature if the space of $L^{2}$ harmonic one-forms is
non-trivial then the first eigenvalue of the Laplacian satisfies $\lambda_{1}\left(M^{n}\right)\leq\frac{b}{\frac{n}{n-1}-a}$.

This theorem will be used to prove the vanishing of $L^{2}$ harmonic
one-forms on complete stable minimal hypersurfaces with a certain
lower bound of the first eigenvalue of the Laplacian in Riemannian
manifolds with $\overline{\biric^{a}}\geq-b$ (Theorem \ref{lambda applic}).
When the Riemannian manifold has dimension at most $7$ a similar
result holds for hypersurfaces not necessarily minimal (Theorem \ref{lambda applic 2}).
In particular, this gives an explicit upper bound of the first eigenvalue
of the Laplacian of minimal and non-minimal stable hypersurfaces with
a non-trivial space of $L^{2}$ harmonic one-forms in the hyperbolic
space (Corollary \ref{lambda applic cor} and Corollary \ref{lambda applic 2 cor}).

This paper also contains the following ridigity result: in Theorem
\ref{lambda cor} with the same assumption of the Ricci curvature
if $\lambda_{1}\left(M^{n}\right)=\frac{b}{\frac{n}{n-1}-a}$ and
the space of $L^{2}$ harmonic one-forms is non-trivial then the universal
cover of $M^{n}$ splits (Theorem \ref{split}).

The main difference between this paper and previous works is that
the results here hold without assumptions of the weight function.
Two important advances are Theorem \ref{vanish cor} which improves
Lam's theorem and Theorem \ref{lambda cor} which proves the vanishing
of $L^{2}$ harmonic forms for more general lower bounds of the Ricci
curvature. Another novelty is the possibility of applying results
from the theory of Riemannian manifolds with a weighted Poincaré inequality
to the theory of stable hypersurfaces, which makes the proofs of many
results short and clear.

As in previous works the proofs of the vanishing theorems of this
paper begin with the well known Bochner-Weitzenböck formula. The key
differences here are Lemma \ref{lemma vanish} and Lemma \ref{lemma lambda},
which may be of independent interest. The proofs of the theorems for
stable hypersurfaces use the following simple idea: a stable hypersurface
satisfies a weighted Poincaré inequality, so to apply Theorem \ref{vanish cor}
or Theorem \ref{lambda cor} to the hypersurface it suffices to show
that the Ricci curvature of the hypersurface satisfies a certain lower
bound.

Notice that Theorem \ref{vanish cor} and Theorem \ref{lambda cor}
are actually corollaries of vanishing theorems for harmonic forms
of any degree (Theorem \ref{vanish} and Theorem \ref{lambda} respectively).

This work is part of the author's Ph.D. thesis, written under the
supervision of Detang Zhou at Universidade Federal Fluminense.

\section{Vanishing of $L^{2}$ harmonic forms}

For a Riemannian manifold $M$ the Hodge Laplacian is defined by
\[
\Delta=-\left(dd^{*}+d^{*}d\right).
\]
The space of $L^{2}$ harmonic $k$-forms is the set of all $k$-forms
$\omega$ on $M$ such that
\[
\Delta\omega=0
\]
and
\[
\int_{M}\left|\omega\right|^{2}<\infty.
\]
For a local orthonormal frame $e_{1},\dots,e_{n}$ on $M$ with dual
coframe $e^{1},\dots,e^{n}$ the curvature operator acting on forms
is defined by 
\[
\mathcal{R}=\sum_{i,j}e^{i}\wedge\iota_{e_{j}}R\left(e_{i},e_{j}\right)
\]
with $R\left(e_{i},e_{j}\right)=\nabla_{i}\nabla_{j}-\nabla_{j}\nabla_{i}$.
For a $k$-form $\omega$ on $M$ the following identity holds
\begin{equation}
\frac{1}{2}\Delta\left|\omega\right|^{2}=\left|\nabla\omega\right|^{2}+\left\langle \Delta\omega,\omega\right\rangle +\left\langle \mathcal{R}\omega,\omega\right\rangle .\label{eq:bochner formula}
\end{equation}
This identity is known as the Bochner-Weitzenböck formula \cite[Lemma 3.4]{LBOOK}.
For a closed and co-closed $k$-form $\omega$ on $M^{n}$ the following
inequality holds 
\begin{equation}
\left|\nabla\omega\right|^{2}\geq C_{n,k}\left|\nabla\left|\omega\right|\right|^{2}\label{eq:des kato}
\end{equation}
with
\[
C_{n,k}=\begin{cases}
1+\frac{1}{n-k}, & 1\leq k\leq n/2,\\
1+\frac{1}{k}, & n/2\leq k\leq n-1.
\end{cases}
\]
This inequality is known as the refined Kato inequality for $L^{2}$
harmonic forms \cite{CGH2000,CZ2012}.

The proof of Theorem \ref{vanish} relies on the Bochner-Weitzenböck
formula, the refined Kato inequality for $L^{2}$ harmonic forms and
the following lemma, which may be of independent interest.
\begin{lem}
\label{lemma vanish}Suppose a complete Riemannian manifold $M$ satisfies
a weighted Poincaré inequality with weight function $q$. Assume a
smooth function $f$ on $M$ satisfies 
\begin{equation}
f\Delta f\geq A\left|\nabla f\right|^{2}-aqf^{2}\label{eq: lemma vanish 1}
\end{equation}
and 
\[
\int_{M}f^{2}<\infty
\]
for some $0<a<1+A$. Then $f$ is constant. Moreover, if $f$ is not
identically zero then the volume of $M$ is finite and $q$ is identically
zero.\end{lem}
\begin{proof}
First we prove that $f$ is constant. Take a cutoff function $\phi$
on $M$ such that 
\[
\phi=\begin{cases}
1 & \text{in }B_{R},\\
0 & \text{in }M\setminus B_{2R},
\end{cases}
\]
\[
0\leq\phi\leq1\,\,\,\,\,\text{in }B_{2R}\setminus B_{R}
\]
and
\[
\left|\nabla\phi\right|\leq\frac{2}{R}\,\,\,\,\,\text{in }B_{2R}\setminus B_{R}.
\]
Here $B_{R}$ is the open ball with center at a fixed point of $M$
and radius $R$. Multiplying inequality (\ref{eq: lemma vanish 1})
by $\phi^{2}$ and integrating by parts gives
\[
\left(1+A\right)\int_{M}\left|\nabla f\right|^{2}\phi^{2}\leq a\int_{M}qf^{2}\phi^{2}-2\int_{M}f\phi\left\langle \nabla f,\nabla\phi\right\rangle .
\]
Putting $f\phi$ in the weighted Poincaré inequality (\ref{eq:wpi})
yields
\[
\int_{M}q\left(f\phi\right)^{2}\leq\int_{M}f^{2}\left|\nabla\phi\right|^{2}+\int_{M}\left|\nabla f\right|^{2}\phi^{2}+2\int_{M}f\phi\left\langle \nabla f,\nabla\phi\right\rangle .
\]
Combining these two inequalities gives
\[
\left(1+A-a\right)\int_{M}\left|\nabla f\right|^{2}\phi^{2}\leq a\int_{M}f^{2}\left|\nabla\phi\right|^{2}+2\left(a-1\right)\int_{M}f\phi\left\langle \nabla f,\nabla\phi\right\rangle .
\]
Fix $\epsilon>0$. By Cauchy\textendash{}Schwarz and Young's inequalities
we have
\[
\left(1+A-a-\epsilon\left|a-1\right|\right)\int_{M}\left|\nabla f\right|^{2}\phi^{2}\leq\left(a+\frac{\left|a-1\right|}{\epsilon}\right)\int_{M}f^{2}\left|\nabla\phi\right|^{2}.
\]
Note that $1+A-a-\epsilon\left|a-1\right|>0$ for a sufficiently small
$\epsilon>0$. Since $\int_{M}f^{2}<\infty$ sending $R\to\infty$
and using the monotone convergence theorem gives
\[
\left(1+A-a-\epsilon\left|a-1\right|\right)\int_{M}\left|\nabla f\right|^{2}\leq0.
\]
This proves that $f$ is constant.

Now we prove that if $f$ is not identically zero then the volume
of $M$ is finite and $q$ is identically zero. Assume $f\neq0$.
Then the volume of $M$ is finite because
\[
\vol\left(M\right)=\frac{\int_{M}f^{2}}{f^{2}}<\infty.
\]
Putting $f$ in inequality (\ref{eq: lemma vanish 1}) shows that
$aqf^{2}\geq0$, which implies that $q\geq0$. Putting $\phi$ in
the weighted Poincaré inequality (\ref{eq:wpi}) gives
\[
\int_{M}q\phi^{2}\leq\frac{4\vol\left(M\right)}{R^{2}}.
\]
Sending $R\to\infty$ and using the monotone convergence theorem yields
\[
\int_{M}q\leq0.
\]
This proves that $q$ is identically zero.
\end{proof}
We prove the following vanishing theorem for $L^{2}$ harmonic $k$-forms
on complete Riemannian manifolds with a weighted Poincaré inequality.
\begin{thm}
\label{vanish}Suppose a complete Riemannian manifold $M^{n}$ satisfies
a weighted Poincaré inequality with weight function $q$ and the curvature
operator satisfies
\[
\left\langle \mathcal{R}\omega,\omega\right\rangle \geq-aq\left|\omega\right|^{2}
\]
for all $k$-forms $\omega$ on $M^{n}$, where $0<a<C_{n,k}$. Assume
that at least one of the following conditions hold: (1) the volume
of $M^{n}$ is infinite; (2) $q$ is not identically zero. Then the
space of $L^{2}$ harmonic $k$-forms on $M^{n}$ is trivial.\end{thm}
\begin{proof}
Fix a $L^{2}$ harmonic $k$-form $\omega$ on $M^{n}$. By Bochner-Weitzenböck
formula (\ref{eq:bochner formula}) we have 
\[
\left|\omega\right|\Delta\left|\omega\right|=\left|\nabla\omega\right|^{2}-\left|\nabla\left|\omega\right|\right|^{2}+\left\langle \mathcal{R}\omega,\omega\right\rangle .
\]
Since $L^{2}$ harmonic forms are closed and co-closed it follows
from the refined Kato inequality for $L^{2}$ harmonic forms (\ref{eq:des kato})
that
\[
f\Delta f\geq\left(C_{n,k}-1\right)\left|\nabla f\right|^{2}-aqf^{2},
\]
where $f=\left|\omega\right|$. Applying Lemma \ref{lemma vanish}
with $A=C_{n,k}-1$ shows that $f$ is constant. If $f$ is not identically
zero it follows from Lemma \ref{lemma vanish} that the volume of
$M^{n}$ is finite and $q$ is identically zero, a contradiction.
This proves that $\omega$ is identically zero.
\end{proof}
Applying Theorem \ref{vanish} with $q=\lambda_{1}\left(M^{n}\right)$
and using the fact that any complete manifold with positive spectrum
has infinite volume proves the following corollary.
\begin{cor}
If a complete Riemannian manifold $M^{n}$ has positive spectrum $\lambda_{1}\left(M^{n}\right)>0$
and the curvature operator satisfies
\[
\left\langle \mathcal{R}\omega,\omega\right\rangle \geq-a\lambda_{1}\left(M^{n}\right)\left|\omega\right|^{2}
\]
for all $k$-forms $\omega$ on $M^{n}$, where $0<a<C_{n,k}$, then
the space of $L^{2}$ harmonic $k$-forms on $M^{n}$ is trivial.
\end{cor}
Taking $k=1$ in this corollary recovers Li and Wang's theorem in
the introduction because for one-forms $\left\langle \mathcal{R}\omega,\omega\right\rangle =\ric\left(\omega,\omega\right)$
and $C_{n,1}=\frac{n}{n-1}$.

We can now prove Theorem \ref{vanish cor}.
\begin{proof}[Proof of Theorem \ref{vanish cor}]
For one-forms $\left\langle \mathcal{R}\omega,\omega\right\rangle =\ric\left(\omega,\omega\right)$
and $C_{n,1}=\frac{n}{n-1}$. Suppose the space of $L^{2}$ harmonic
one-forms is non-trivial. Then by Theorem \ref{vanish} the manifold
has finite volume and non-negative Ricci curvature. However Yau proved
that any complete non-compact Riemannian manifold with non-negative
Ricci curvature has infinite volume \cite{Y1976}. This proves that
the space of $L^{2}$ harmonic one-forms is trivial.
\end{proof}
The proofs of Theorem \ref{lambda} and Theorem \ref{split} rely
on the following lemma, which may be of independent interest.
\begin{lem}
\label{lemma lambda}Suppose a complete Riemannian manifold $M$ satisfies
a weighted Poincaré inequality with weight function $q$. Assume a
smooth function $f$ on $M$ satisfies
\begin{equation}
f\Delta f\geq A\left|\nabla f\right|^{2}-aqf^{2}-bf^{2}\label{eq: lemma vanish 2}
\end{equation}
and 
\[
\int_{M}f^{2}<\infty
\]
for some $0<a<1+A$ and $b>0$. Then
\begin{equation}
\int_{M}\left|\nabla f\right|^{2}\leq\frac{b}{1+A-a}\int_{M}f^{2}.\label{eq: lemma vanish 2-1}
\end{equation}
Moreover, if equality holds in (\ref{eq: lemma vanish 2-1}) then
equality holds in (\ref{eq: lemma vanish 2}).\end{lem}
\begin{proof}
First we show inequality (\ref{eq: lemma vanish 2-1}). Take a cutoff
function $\phi$ on $M$ such that 
\[
\phi=\begin{cases}
1 & \text{in }B_{R},\\
0 & \text{in }M\setminus B_{2R},
\end{cases}
\]
\[
0\leq\phi\leq1\,\,\,\,\,\text{in }B_{2R}\setminus B_{R}
\]
and
\[
\left|\nabla\phi\right|\leq\frac{2}{R}\,\,\,\,\,\text{in }B_{2R}\setminus B_{R}.
\]
Here $B_{R}$ is the open ball with center at a fixed point of $M$
and radius $R$. Multiplying inequality (\ref{eq: lemma vanish 2})
by $\phi^{2}$ and integrating by parts gives
\[
\left(1+A\right)\int_{M}\left|\nabla f\right|^{2}\phi^{2}\leq a\int_{M}qf^{2}\phi^{2}+b\int_{M}f^{2}\phi^{2}-2\int_{M}f\phi\left\langle \nabla f,\nabla\phi\right\rangle .
\]
Putting $f\phi$ in the weighted Poincaré inequality (\ref{eq:wpi})
yields
\begin{equation}
\int_{M}q\left(f\phi\right)^{2}\leq\int_{M}f^{2}\left|\nabla\phi\right|^{2}+\int_{M}\left|\nabla f\right|^{2}\phi^{2}+2\int_{M}f\phi\left\langle \nabla f,\nabla\phi\right\rangle .\label{eq:lemma vanish 2-2}
\end{equation}
Combining these two inequalities gives
\[
\left(1+A-a\right)\int_{M}\left|\nabla f\right|^{2}\phi^{2}\leq b\int_{M}f^{2}\phi^{2}+a\int_{M}f^{2}\left|\nabla\phi\right|^{2}+2\left(a-1\right)\int_{M}f\phi\left\langle \nabla f,\nabla\phi\right\rangle .
\]
Fix $\epsilon>0$. By Cauchy\textendash{}Schwarz and Young's inequalities
we have
\[
\left(1+A-a-\epsilon\left|a-1\right|\right)\int_{M}\left|\nabla f\right|^{2}\phi^{2}\leq b\int_{M}f^{2}\phi^{2}+\left(a+\frac{\left|a-1\right|}{\epsilon}\right)\int_{M}f^{2}\left|\nabla\phi\right|^{2}.
\]
Note that $1+A-a-\epsilon\left|a-1\right|>0$ for all sufficiently
small $\epsilon>0$. Since $\int_{M}f^{2}<\infty$ sending $R\to\infty$,
using the monotone convergence theorem and then sending $\epsilon\to0$
proves inequality (\ref{eq: lemma vanish 2-1}).

Now we assume that equality holds in (\ref{eq: lemma vanish 2-1}).
Multiplying inequality (\ref{eq: lemma vanish 2}) by $\phi^{2}$
and integrating gives
\begin{align*}
0 & \leq\int_{M}\left(f\Delta f-A\left|\nabla f\right|^{2}+aqf^{2}+bf^{2}\right)\phi^{2}\\
 & =-\left(1+A\right)\int_{M}\left|\nabla f\right|^{2}\phi^{2}+b\int_{M}f^{2}\phi^{2}-2\int_{M}f\phi\left\langle \nabla f,\nabla\phi\right\rangle +a\int_{M}qf^{2}\phi^{2}.
\end{align*}
Combining this with inequality (\ref{eq:lemma vanish 2-2}) and using
the Cauchy\textendash{}Schwarz inequality yields
\begin{align}
0 & \leq\int_{M}\left(f\Delta f-A\left|\nabla f\right|^{2}+aqf^{2}+bf^{2}\right)\phi^{2}\label{eq: lemma vanish 2-3}\\
 & \leq-\left(1+A-a\right)\int_{M}\left|\nabla f\right|^{2}\phi^{2}+b\int_{M}f^{2}\phi^{2}\nonumber \\
 & +a\int_{M}f^{2}\left|\nabla\phi\right|^{2}+2\left|a-1\right|\left(\int_{M}f^{2}\left|\nabla\phi\right|^{2}\right)^{\frac{1}{2}}\left(\int_{M}\left|\nabla f\right|^{2}\phi^{2}\right)^{\frac{1}{2}}.\nonumber 
\end{align}
Since the quantity being integrated in (\ref{eq: lemma vanish 2-3})
is non-negative and $\int_{M}f^{2}<\infty$ and $\int_{M}\left|\nabla f\right|^{2}<\infty$,
sending $R\to\infty$ and using the monotone convergence theorem gives
\begin{align}
0 & \leq\int_{M}\left(f\Delta f-A\left|\nabla f\right|^{2}+aqf^{2}+bf^{2}\right)\label{eq: lemma vanish 2-4}\\
 & \leq-\left(1+A-a\right)\int_{M}\left|\nabla f\right|^{2}+b\int_{M}f^{2}\nonumber \\
 & =0.\nonumber 
\end{align}
Since the quantity being integrated in (\ref{eq: lemma vanish 2-4})
is non-negative it follows that equality holds in (\ref{eq: lemma vanish 2}).
\end{proof}
The proof of Theorem \ref{vanish} also uses the following straightforward
lemma.
\begin{lem}
\label{lemma lambda 2}If a smooth function $f$ on a complete Riemannian
manifold $M$ satisfies 
\[
\int_{M}f^{2}<\infty,
\]
then 
\[
\lambda_{1}\left(M\right)\int_{M}f^{2}\leq\int_{M}\left|\nabla f\right|^{2}.
\]
\end{lem}
\begin{proof}
Take a cutoff function $\phi$ on $M$ such that 
\[
\phi=\begin{cases}
1 & \text{in }B_{R},\\
0 & \text{in }M\setminus B_{2R},
\end{cases}
\]
\[
0\leq\phi\leq1\,\,\,\,\,\text{in }B_{2R}\setminus B_{R}
\]
and
\[
\left|\nabla\phi\right|\leq\frac{2}{R}\,\,\,\,\,\text{in }B_{2R}\setminus B_{R}.
\]
Fix $\epsilon>0$. Then 
\begin{align*}
\lambda_{1}\left(M\right)\int_{M}\left(f\phi\right)^{2} & \leq\int_{M}f^{2}\left|\nabla\phi\right|^{2}+\int_{M}\left|\nabla f\right|^{2}\phi^{2}+2\int_{M}f\phi\left\langle \nabla f,\nabla\phi\right\rangle \\
 & \leq\left(1+\epsilon\right)\int_{M}\left|\nabla f\right|^{2}\phi^{2}+\left(1+\frac{1}{\epsilon}\right)\int_{M}f^{2}\left|\nabla\phi\right|^{2}.
\end{align*}
Since $\int_{M}f^{2}<\infty$ sending $R\to\infty$, using the monotone
convergence theorem and then sending $\epsilon\to0$ proves the result.
\end{proof}
We prove the following vanishing theorem, which shows that Theorem
\ref{vanish} holds with a more general lower bound of the curvature
operator when the first eigenvalue of the Laplacian satisfies a certain
lower bound.
\begin{thm}
\label{lambda}Suppose a complete Riemannian manifold $M^{n}$ satisfies
a weighted Poincaré inequality with weight function $q$ and the curvature
operator satisfies
\[
\left\langle \mathcal{R}\omega,\omega\right\rangle \geq-aq\left|\omega\right|^{2}-b\left|\omega\right|^{2}
\]
for all $k$-forms $\omega$ on $M^{n}$, where $0<a<C_{n,k}$ and
$b>0$. Assume the first eigenvalue of the Laplacian satisfies
\[
\lambda_{1}\left(M^{n}\right)>\frac{b}{C_{n,k}-a}.
\]
Then the space of $L^{2}$ harmonic $k$-forms on $M^{n}$ is trivial.\end{thm}
\begin{proof}
Fix a $L^{2}$ harmonic $k$-form $\omega$ on $M^{n}$. As in the
proof of Theorem \ref{vanish} we have 
\[
f\Delta f\geq\left(C_{n,k}-1\right)\left|\nabla f\right|^{2}-aqf^{2}-bf^{2},
\]
where $f=\left|\omega\right|$. Applying Lemma \ref{lemma lambda}
with $A=C_{n,k}-1$ gives 
\[
\int_{M^{n}}\left|\nabla f\right|^{2}\leq\frac{b}{C_{n,k}-a}\int_{M^{n}}f^{2}.
\]
By Lemma \ref{lemma lambda 2} we have 
\[
\lambda_{1}\left(M^{n}\right)\int_{M^{n}}f^{2}\leq\int_{M^{n}}\left|\nabla f\right|^{2}.
\]
Combining the last two inequalities yields 
\[
\lambda_{1}\left(M^{n}\right)\int_{M^{n}}f^{2}\leq\frac{b}{C_{n,k}-a}\int_{M^{n}}f^{2}.
\]
If $\omega$ is not identically zero then 
\[
\lambda_{1}\left(M^{n}\right)\leq\frac{b}{C_{n,k}-a},
\]
a contradiction. This proves that $\omega$ is identically zero.
\end{proof}
Applying Theorem \ref{lambda} to one-forms proves Theorem \ref{lambda cor}.
\begin{proof}[Proof of Theorem \ref{lambda cor}]
For one-forms $\left\langle \mathcal{R}\omega,\omega\right\rangle =\ric\left(\omega,\omega\right)$
and $C_{n,1}=\frac{n}{n-1}$.
\end{proof}

To study the ridigity in Theorem \ref{lambda cor} we use the following
lemma.
\begin{lem}[{\cite[Lemma 4.1]{LW2006}}]
Let $M^{n}$ be a complete Riemannian manifold of dimension $n\geq2$.
Assume that the Ricci curvature of $M$ satisfies the lower bound
\[
\ric_{M}\left(x\right)\geq-\left(n-1\right)\tau\left(x\right)
\]
for all $x\in M$. Suppose $f$ is a nonconstant harmonic function
defined on $M$. Then the function $\left|\nabla f\right|$ must satisfy
the differential inequality
\[
\Delta\left|\nabla f\right|\geq-\left(n-1\right)\tau\left|\nabla f\right|+\frac{\left|\nabla\left|\nabla f\right|\right|^{2}}{\left(n-1\right)\left|\nabla f\right|}
\]
in the weak sense. Moreover, if equality holds, then $M$ is given
by $M=\mathbb{R}\times N^{n-1}$ with the warped product metric
\[
ds_{M}^{2}=dt^{2}+\eta\left(t\right)^{2}ds_{N}^{2}
\]
for some positive function $\eta(t)$, and some manifold $N^{n-1}$.
In this case, $\tau\left(t\right)$ is a function of $t$ alone satisfying
\[
\eta''\left(t\right)\eta^{-1}\left(t\right)=\tau\left(t\right).
\]

\end{lem}
We prove the following ridigity theorem. The idea is to combine the
lemma above with the equality conclusion of Lemma \ref{lemma lambda}.
\begin{thm}
\label{split}Suppose a complete Riemannian manifold $M^{n}$ satisfies
a weighted Poincaré inequality with weight function $q$ and the Ricci
curvature satisfies 
\[
\ric\geq-aq-b
\]
for some $0<a<\frac{n}{n-1}$ and $b>0$. Assume the first eigenvalue
of the Laplacian satisfies 
\[
\lambda_{1}\left(M^{n}\right)=\frac{b}{\frac{n}{n-1}-a},
\]
and the space of $L^{2}$ harmonic one-forms on $M^{n}$ is non-trivial.
Then the universal cover of $M^{n}$ splits as $\widetilde{M}^{n}=\mathbb{R}\times N^{n-1}$
with the warped product metric
\[
g_{\widetilde{M}^{n}}=dt^{2}+\eta\left(t\right)^{2}g_{N^{n-1}}
\]
for some positive function $\eta\left(t\right)$ and some hypersurface
$N^{n-1}$ in $\widetilde{M}^{n}$. In this case, $q$ is a function
of $t$ alone satisfying
\[
\frac{\eta''\left(t\right)}{\eta\left(t\right)}=\frac{1}{n-1}\left(aq+b\right).
\]
\end{thm}
\begin{proof}
Take a non-vanishing $L^{2}$ harmonic one-form $\omega$ on $M^{n}$.
By the proof of Theorem \ref{lambda} we have
\[
f\Delta f\geq\frac{1}{n-1}\left|\nabla f\right|^{2}-aqf^{2}-bf^{2}
\]
and
\[
\int_{M^{n}}\left|\nabla f\right|^{2}=\frac{b}{\frac{n}{n-1}-a}\int_{M^{n}}f^{2},
\]
where $f=\left|\omega\right|$. By Lemma \ref{lemma lambda} we have
\[
f\Delta f=\frac{1}{n-1}\left|\nabla f\right|^{2}-aqf^{2}-bf^{2}.
\]
Lift the metric of $M^{n}$ to the universal cover $\widetilde{M}^{n}$
and lift $\omega$ to a harmonic one-form $\widetilde{\omega}$ on
$\widetilde{M}^{n}$. Since $\widetilde{M}^{n}$ is simply connected,
there is a smooth function $h$ on $\widetilde{M}^{n}$ such that
$dh=\widetilde{\omega}$. This shows that $h$ is a non-constant harmonic
function on $\widetilde{M}^{n}$ such that
\[
\left|dh\right|\Delta\left|dh\right|=\frac{1}{n-1}\left|\nabla\left|dh\right|\right|^{2}-(aq+b)\left|dh\right|^{2}.
\]
The conclusion follows from the lemma above.\end{proof}
\begin{rem}
The results of this section can be improved using refined Kato inequalities
for $L^{2}$ harmonic forms on Kähler manifolds (\cite[Theorem 4.2]{L2010}
and \cite[Theorem 3.1]{KLZ2008}): a $L^{2}$ harmonic one-form $\omega$
on a complete Kähler manifold satisfies 
\[
\left|\nabla\omega\right|^{2}\geq2\left|\nabla\left|\omega\right|\right|^{2}.
\]
Using this inequality it is not difficult to show that Theorem \ref{vanish cor}
and Theorem \ref{lambda cor} hold with ``$0<a<\frac{n}{n-1}$''
and ``Riemannian manifold'' replaced by ``$0<a<2$'' and ``Kähler
manifold'' respectively.
\end{rem}

\section{Applications to stable hypersurfaces}

For a hypersurface $M^{n}$ in a Riemannian manifold $\overline{M}^{n+1}$
the stability operator (or Jacobi operator) is defined by
\[
L=\Delta+\left|A\right|^{2}+\overline{\ric}\left(\nu,\nu\right),
\]
where $\Delta$ is the Laplacian of $M^{n}$, $A$ is the second fundamental
form and $\overline{\ric}\left(\nu,\nu\right)$ is the Ricci curvature
of $\overline{M}^{n+1}$ in the normal direction. The hypersurface
$M^{n}$ is called stable if the first eigenvalue of the stability
operator is non-negative, in other words
\[
0\leq\lambda_{1}\left(L\right)=\inf_{\phi\in C_{c}^{\infty}\left(M^{n}\right)}\frac{\int_{M^{n}}\left(-L\phi\right)\phi}{\int_{M^{n}}\phi^{2}}.
\]
This shows that a stable hypersurface satisfies a weighted Poincaré
inequality with weight function 
\[
q=\left|A\right|^{2}+\overline{\ric}\left(\nu,\nu\right).
\]
For orthonormal vector fields $X$ and $Y$ on $\overline{M}^{n+1}$
the Bi-Ricci$^{a}$ curvature is defined by 
\[
\overline{\biric^{a}}\left(X,Y\right)=\overline{\ric}\left(X,X\right)+a\overline{\ric}\left(Y,Y\right)-\overline{K}\left(X,Y\right),
\]
where $\overline{K}$ is the sectional curvature of $\overline{M}^{n+1}$
and $a$ is a constant. For $a=1$ the Bi-Ricci$^{a}$ curvature is
equal to the Bi-Ricci curvature defined by Shen and Ye \cite{SY1996}.
Notice that if the sectional curvature is non-negative then the Bi-Ricci
curvature is non-negative.

It is an interesting problem to study the geometry and topology of
stable minimal hypersurfaces in Riemannian manifolds with a certain
non-negative curvature. Fischer-Colbrie and Schoen classified complete
stable minimal surfaces in three-dimensional Riemannian manifolds
with non-negative scalar curvature \cite{FCS1980}. Palmer proved
that the space of $L^{2}$ harmonic one-forms on a complete stable
minimal hypersurface in $\mathbb{R}^{n+1}$ is trivial \cite{P1991}.
Miyaoka and Tanno extended this result to hypersurfaces in Riemannian
manifolds with non-negative sectional curvature and non-negative Bi-Ricci
curvature respectively \cite{M1993,T1996}. Cao, Shen and Zhu proved
that a complete stable minimal hypersurface in $\mathbb{R}^{n+1}$
has only one end \cite{CSZ1997}. Li and Wang proved that for a complete
stable minimal hypersurface properly immersed in a Riemannian manifold
with non-negative sectional curvature either the hypersurface has
only one end or it is totally geodesic with a certain decomposition
\cite{LW2004}. Cheng proved that a complete stable minimal hypersurface
in a Riemannian manifold of dimension at most $6$ and positive Bi-Ricci$^{a}$
curvature has only one end \cite{C2008}.

We prove the following vanishing theorem for $L^{2}$ harmonic one-forms
on complete stable minimal hypersurfaces in Riemannian manifolds with
non-negative Bi-Ricci$^{a}$ curvature.
\begin{thm}
\label{vanish applic}For a complete non-compact stable minimal hypersurface
$M^{n}$ in a Riemannian manifold $\overline{M}^{n+1}$ with non-negative
Bi-Ricci$^{a}$ curvature for some $\frac{n-1}{n}\leq a<\frac{n}{n-1}$,
the space of $L^{2}$ harmonic one-forms on $M^{n}$ is trivial.\end{thm}
\begin{proof}
By assumption $M^{n}$ satisfies a weighted Poincaré inequality with
weight function 
\[
q=\left|A\right|^{2}+\overline{\ric}\left(\nu,\nu\right).
\]
By Theorem \ref{vanish cor} to complete the proof it suffices to
show that the Ricci curvature of $M^{n}$ satisfies 
\[
\ric\geq-aq.
\]
The proof is adapted from \cite{LW2004}. Take a local orthonormal
frame $e_{1},\dots,e_{n}$ on $M^{n}$ diagonalizing the second fundamental
form, in other words $A\left(e_{i},e_{j}\right)=\lambda_{i}\delta_{ij}$.
By the Gauss equation for $i\neq j$ we have 
\[
K\left(e_{i},e_{j}\right)=\overline{K}\left(e_{i},e_{j}\right)+\lambda_{i}\lambda_{j}.
\]
Since $M^{n}$ is minimal we have 
\[
\lambda_{1}+\sum_{i=2}^{n}\lambda_{i}=0.
\]
This shows that
\[
\ric\left(e_{1},e_{1}\right)=\overline{\ric}\left(e_{1},e_{1}\right)-\overline{K}\left(e_{1},\nu\right)-\lambda_{1}^{2}.
\]
This implies that 
\begin{align*}
\ric\left(e_{1},e_{1}\right)+aq & =\overline{\biric^{a}}\left(e_{1},\nu\right)+a\left|A\right|^{2}-\lambda_{1}^{2}\\
 & \geq a\left(\lambda_{1}^{2}+\sum_{i=2}^{n}\lambda_{i}^{2}\right)-\lambda_{1}^{2}\\
 & \geq a\left(\lambda_{1}^{2}+\frac{\left(\sum_{i=2}^{n}\lambda_{i}\right)^{2}}{n-1}\right)-\lambda_{1}^{2}\\
 & =\left(\frac{na}{n-1}-1\right)\lambda_{1}^{2}\\
 & \geq0.
\end{align*}
This proves the theorem.
\end{proof}
Taking $a=1$ in this theorem recovers Tanno's theorem \cite{T1996}.
We hope to find a good example of a Riemannian manifold with non-negative
Bi-Ricci$^{a}$ curvature for some $\frac{n-1}{n}\leq a<\frac{n}{n-1}$
but without non-negative Bi-Ricci curvature. Li and Wang's theorem
has a stronger conclusion but assuming the Riemannian manifold with
non-negative sectional curvature and the hypersurface properly immersed
\cite{LW2004}. Cheng's theorem has a stronger conclusion but assuming
the Riemannian manifold with positive Bi-Ricci$^{a}$ curvature and
$3\leq n\leq5$ with $\frac{2}{3}\leq a\leq2$ for $n=3$ and $\frac{n-1}{n}\leq a<\frac{4}{n-1}$
for $n=4,5$ \cite{C2008}.

There has been some interest in studying the topology at infinity
of stable constant mean curvature hypersurfaces \cite{C2000,CCZ2008}.

We prove the following vanishing theorem for $L^{2}$ harmonic one-forms
on complete stable hypersurfaces, not necessarily minimal, in Riemannian
manifolds with non-negative Bi-Ricci$^{a}$ and dimension at most
$7$.
\begin{thm}
\label{vanish applic 2}For a complete non-compact stable hypersurface
$M^{n}$ in a Riemannian manifold $\overline{M}^{n+1}$ with $2\leq n\leq6$
and non-negative Bi-Ricci$^{a}$ curvature for some $\frac{\sqrt{n-1}}{2}\leq a<\frac{n}{n-1}$,
the space of $L^{2}$ harmonic one-forms on $M^{n}$ is trivial.\end{thm}
\begin{proof}
Note that $\frac{\sqrt{n-1}}{2}<\frac{n}{n-1}$ for $2\leq n\leq6$.
As in Theorem \ref{vanish applic} to complete the proof it suffices
to show that the Ricci curvature of $M^{n}$ satisfies 
\[
\ric\geq-aq.
\]
As in Theorem \ref{vanish applic} we have
\[
\ric\left(e_{1},e_{1}\right)=\overline{\ric}\left(e_{1},e_{1}\right)-\overline{K}\left(e_{1},\nu\right)+\lambda_{1}\sum_{i=2}^{n}\lambda_{i}.
\]
This shows that
\begin{align*}
\ric\left(e_{1},e_{1}\right)+aq & =\overline{\biric^{a}}\left(e_{1},\nu\right)+\lambda_{1}\sum_{i=2}^{n}\lambda_{i}+a\left|A\right|^{2}\\
 & \geq\lambda_{1}\sum_{i=2}^{n}\lambda_{i}+a\sum_{i=1}^{n}\lambda_{i}^{2}.
\end{align*}
If $\lambda_{1}=0$ then 
\[
\ric\left(e_{1},e_{1}\right)+aq\geq0.
\]
If $\lambda_{1}\neq0$ then
\begin{align*}
\ric\left(e_{1},e_{1}\right)+aq & \geq\lambda_{1}\sum_{i=2}^{n}\lambda_{i}+a\sum_{i=1}^{n}\lambda_{i}^{2}\\
 & =\lambda_{1}^{2}\left(\sum_{i=2}^{n}\left(\frac{\sqrt{a}\lambda_{i}}{\lambda_{1}}+\frac{1}{2\sqrt{a}}\right)^{2}+a-\frac{n-1}{4a}\right)\\
 & \geq\lambda_{1}^{2}\left(a-\frac{n-1}{4a}\right)\\
 & =0.
\end{align*}
This proves the theorem.
\end{proof}
Notice that the proof of the theorem implies following result: For
a hypersurface $M^{n}$ in a Riemannian manifold $\overline{M}^{n+1}$
and $a\geq\frac{\sqrt{n-1}}{2}$ we have 
\[
\ric\left(e_{1},e_{1}\right)+a\left(\left|A\right|^{2}+\overline{\ric}\left(\nu,\nu\right)\right)\geq\overline{\biric^{a}}\left(e_{1},\nu\right).
\]

After the paper was submitted the author became aware that this result
was proved for hypersurfaces in Riemannian manifolds with non-negative
sectional curvature by Kim and Yun \cite{KY2013} and Dung and Seo
\cite{DS2015}. Our technique to find the lower bound of the Ricci
curvature of the hypersurface is more elementary.

There has been some interest in finding estimates for eigenvalues
of the Laplacian of minimal hypersurfaces in the hyperbolic space.
Cheung and Leung proved that for a complete minimal hypersurface $M^{n}$
in the hyperbolic space $\mathbb{H}^{n+1}$ the first eigenvalue of
the Laplacian of $M^{n}$ satisfies the lower bound $\lambda_{1}\left(M^{n}\right)\geq\frac{1}{4}\left(n-1\right)^{2}$
\cite{CL2001}. Candel proved that for a simply connected stable minimal
surface $M^{2}$ in the three-dimensional hyperbolic space $\mathbb{H}^{3}$
the first eigenvalue of the Laplacian of $M^{2}$ satisfies the upper
bound $\lambda_{1}\left(M^{2}\right)\leq\frac{4}{3}$ \cite{C2007}.
Seo proved that for a complete stable minimal hypersurface $M^{n}$
in the hyperbolic space $\mathbb{H}^{n+1}$ if the norm of the second
fundamental form belongs to $L^{2}$ then the first eigenvalue of
the Laplacian of $M^{n}$ satisfies the upper bound $\lambda_{1}\left(M^{n}\right)\leq n^{2}$
\cite{S2011}.

We prove the following vanishing theorem for $L^{2}$ harmonic one-forms
on complete stable minimal hypersurfaces in Riemannian manifolds with
$\overline{\biric^{a}}\geq-b$.
\begin{thm}
\label{lambda applic}For a complete  stable minimal hypersurface
$M^{n}$ in a Riemannian manifold $\overline{M}^{n+1}$ with $\overline{\biric^{a}}\geq-b$
for some $\frac{n-1}{n}\leq a<\frac{n}{n-1}$ and $b>0$, if the first
eigenvalue of the Laplacian of $M^{n}$ satisfies 
\[
\lambda_{1}\left(M^{n}\right)>\frac{b}{\frac{n}{n-1}-a}
\]
then the space of $L^{2}$ harmonic one-forms on $M^{n}$ is trivial.\end{thm}
\begin{proof}
By assumption $M^{n}$ satisfies a weighted Poincaré inequality with
weight function 
\[
q=\left|A\right|^{2}+\overline{\ric}\left(\nu,\nu\right).
\]
As in Theorem \ref{vanish applic} the Ricci curvature of $M^{n}$
satisfies 
\[
\ric\geq-aq-b.
\]
The conclusion follows from Theorem \ref{lambda cor}.
\end{proof}
This theorem can be used to obtain an explicit upper bound of the
first eigenvalue of the Laplacian of stable minimal hypersurfaces
in Riemannian manifolds with constant negative sectional curvature.
\begin{cor}
\label{lambda applic cor}For a complete stable minimal hypersurface
$M^{n}$ in a Riemannian manifold $\overline{M}^{n+1}$ with constant
sectional curvature $-\overline{K}^{2}\neq0$, if the space of $L^{2}$
harmonic one-forms on $M^{n}$ is non-trivial then the first eigenvalue
of the Laplacian of $M^{n}$ satisfies 
\[
\lambda_{1}\left(M^{n}\right)\leq\overline{K}^{2}\left(\frac{2n\left(n-1\right)^{2}}{2n-1}\right).
\]
In particular, this estimate holds when $M^{n}$ has at least two
ends.\end{cor}
\begin{proof}
The first conclusion follows by taking $a=\frac{n-1}{n}$ and $b=\overline{K}^{2}\left(n+an-1\right)$
in Theorem \ref{lambda applic}. Now assume that $M^{n}$ has at least
two ends. Since $\overline{M}^{n+1}$ has negative constant sectional
curvature and $M^{n}$ is minimal it follows from \cite{HS1974} that
$M^{n}$ satisfies a certain Sobolev inequality. Using Holder's inequality
gives a Sobolev inequality as in \cite[Lemma 20.12]{LBOOK}, which
implies the non-existence of parabolic ends in $M^{n}$. This shows
that $M^{n}$ has at least two non-parabolic ends. It follows from
\cite{LT1992} that $M^{n}$ admits a non-trivial $L^{2}$ harmonic
one-form. The second conclusion follows from the first conclusion
of the theorem.
\end{proof}
Taking $\overline{M}^{n+1}=\mathbb{H}^{3}$ in this corollary recovers
the upper bound obtained by Candel \cite{C2007} but assuming the
space of $L^{2}$ harmonic one-forms on $M^{2}$ non-trivial. Taking
$\overline{M}^{n+1}=\mathbb{H}^{n+1}$ in the corollary gives the
upper bound $\lambda_{1}\left(M^{n}\right)\leq\frac{2n(n-1)^{2}}{2n-1}$,
which is better than upper bound $\lambda_{1}\left(M^{n}\right)\leq n^{2}$
obtained by Seo \cite{S2011}. Notice that Seo's theorem holds without
assuming the space of $L^{2}$ harmonic one-forms non-trivial but
assuming the norm of the second fundamental form $L^{2}$-integrable.

We prove the following vanishing theorem for $L^{2}$ harmonic one-forms
on complete stable hypersurfaces, not necessarily minimal, in Riemannian
manifolds with $\overline{\biric^{a}}\geq-b$ and dimension at most
$7$.
\begin{thm}
\label{lambda applic 2}For a complete stable hypersurface $M^{n}$
in a Riemannian manifold $\overline{M}^{n+1}$ with $2\leq n\leq6$
and $\overline{\biric^{a}}\geq-b$ for some $\frac{\sqrt{n-1}}{2}\leq a<\frac{n}{n-1}$
and $b>0$, if the first eigenvalue of the Laplacian of $M^{n}$ satisfies
\[
\lambda_{1}\left(M^{n}\right)>\frac{b}{\frac{n}{n-1}-a}
\]
then the space of $L^{2}$ harmonic one-forms on $M^{n}$ is trivial.\end{thm}
\begin{proof}
Note that $\frac{\sqrt{n-1}}{2}<\frac{n}{n-1}$ for $2\leq n\leq6$.
By assumption $M^{n}$ satisfies a weighted Poincaré inequality with
weight function 
\[
q=\left|A\right|^{2}+\overline{\ric}\left(\nu,\nu\right).
\]
As in Theorem \ref{vanish applic 2} the Ricci curvature of $M^{n}$
satisfies 
\[
\ric\geq-aq-b.
\]
The conclusion follows from Theorem \ref{lambda cor}.
\end{proof}
This theorem can be used to obtain an explicit upper bound of the
first eigenvalue of the Laplacian of stable hypersurfaces, not necessarily
minimal, in Riemannian manifolds with constant negative sectional
curvature and dimension at most $7$.
\begin{cor}
\label{lambda applic 2 cor}For a complete stable hypersurface $M^{n}$
in a Riemannian manifold $\overline{M}^{n+1}$ with constant sectional
curvature $-\overline{K}^{2}\neq0$ and dimension $2\leq n\leq6$,
if the space of $L^{2}$ harmonic one-forms on $M^{n}$ is non-trivial
then the first eigenvalue of the Laplacian of $M^{n}$ satisfies 
\[
\lambda_{1}\left(M^{n}\right)\leq\overline{K}^{2}\left(\frac{n-1+\frac{\sqrt{n-1}}{2}n}{\frac{n}{n-1}-\frac{\sqrt{n-1}}{2}}\right).
\]
\end{cor}
\begin{proof}
The conclusion follows by taking $a=\frac{\sqrt{n-1}}{2}$ and $b=\overline{K}^{2}\left(n+an-1\right)$
in Theorem \ref{lambda applic 2}.
\end{proof}
In particular, we obtain the following result.
\begin{cor}
For a complete stable surface $M^{2}$, not necessarily minimal, in
the hyperbolic space $\mathbb{H}^{3}$, if the space of $L^{2}$ harmonic
one-forms on $M^{2}$ is non-trivial then the first eigenvalue of
the Laplacian of $M^{2}$ satisfies 
\[
\lambda_{1}\left(M^{2}\right)\leq\frac{4}{3}.
\]

\end{cor}
This corollary extends the upper bound obtained by Candel \cite{C2007}
to non-minimal surfaces, but assuming the space of $L^{2}$ harmonic
one-forms on $M^{2}$ non-trivial.

\lyxaddress{Matheus Vieira}

\lyxaddress{Departamento de Matemática, Universidade Federal do Espírito Santo,
Vitória, 29075-910, Brazil}

\lyxaddress{e-mail: matheus.vieira@ufes.br}
\end{document}